\documentclass[11pt]{article}
\usepackage{amsmath,amssymb,amsthm,amsfonts,mathrsfs}
\usepackage{appendix}
\usepackage{enumitem,graphics,xcolor,yfonts,colonequals}
\usepackage{stmaryrd}
\usepackage[alphabetic]{amsrefs}
\usepackage[all,cmtip]{xy}
\usepackage[colorlinks,anchorcolor=blue,citecolor=blue,linkcolor=blue,urlcolor =blue,bookmarksopen=true]{hyperref}
\usepackage{csquotes}
\usepackage{fancyhdr}
\fancypagestyle{titlepage}
{
	\fancyhf{}

	\fancyfoot[l]{
	\href{https://zbmath.org/classification/}
		{\emph{2020 Mathematics Subject Classification}}
		14J28, 53D12
		\\
		\emph{Keywords}: Lagrangian submanifolds, special Lagrangian fibrations, K3 surfaces.
	}
}

\AtBeginDocument{%
	\def\MR#1{}
}

\newcommand{\bZ}{\mathbb{Z}}    %Integers
\newcommand{\bQ}{\mathbb{Q}}    %Rational numbers
\newcommand{\bR}{\mathbb{R}}    %Real numbers
\newcommand{\bC}{\mathbb{C}}    %Complex numbers

\newcommand{\bP}{\mathbb{P}}    %Projective space
\newcommand{\cC}{\mathcal{C}}   %Positive cone
\newcommand{\cD}{\mathcal{D}}   %Period domain
\newcommand{\Kthree}{\mathrm{K3}} %Display K3 in roman
\newcommand{\Sp}{\mathrm{Sp}}
\newcommand{\Lag}{\mathrm{Lag}} %Lagrangian lattice
\newcommand{\SLag}{\mathrm{SLag}}   %Special Lagrangian lattice
\newcommand{\cK}{\mathcal{K}}   %Kahler cone
\newcommand{\cQ}{\mathcal{Q}}   %Quadric in H2(K3,C)

\newtheorem*{thm*}{Theorem}
\newtheorem*{prop*}{Proposition}
\newtheorem*{cor*}{Corollary}
\newtheorem{thm}{Theorem}[section]
\newtheorem{prop}[thm]{Proposition}
\newtheorem{cor}[thm]{Corollary}
\newtheorem{lemma}[thm]{Lemma}
\numberwithin{equation}{section}

\theoremstyle{definition}

\newtheorem{rmk}[thm]{Remark}

\begin{document}
\title{\bf Decomposition of Lagrangian classes\\ on K3~surfaces}
\author{Kuan-Wen~Lai, Yu-Shen~Lin, and Luca~Schaffler}
\date{}

%--------------------Contact information
\newcommand{\ContactInfo}{{
% additional braces for segregating \footnotesize
\bigskip\footnotesize

\bigskip
\noindent K.-W.~Lai,
\textsc{Mathematisches Institut der Universit\"at Bonn\\
Endenicher Allee 60, 53121 Bonn, Deutschland}\par\nopagebreak
\noindent\textsc{Email:} \noindent\texttt{kwlai@math.uni-bonn.de}

\bigskip
\noindent Y.-S.~Lin,
\textsc{Department of Mathematics \& Statistics\\
Boston University\\
Boston, MA 02215, USA}\par\nopagebreak
\noindent\textsc{Email:} \texttt{yslin@bu.edu}

\bigskip
\noindent L.~Schaffler,
\textsc{Dipartimento di Matematica e Fisica\\
Universit\`a degli Studi Roma Tre\\
Rome, 00146, Italy}\par\nopagebreak
\noindent\textsc{Email:} \texttt{luca.schaffler@uniroma3.it}
}}

\maketitle
\thispagestyle{titlepage}

%--------------------Abstract
\begin{abstract}
We study the decomposability of a Lagrangian homology class on a K3 surface into a sum of classes represented by special Lagrangian submanifolds, and develop criteria for it in terms of lattice theory.
As a result, we prove the decomposability on an arbitrary K3 surface with respect to the K\"ahler classes in dense subsets of the K\"ahler cone.
Using the same technique, we show that the K\"ahler classes on a K3 surface which admit a special Lagrangian fibration form a dense subset also. This implies that there are infinitely many special Lagrangian $3$-tori in any log Calabi--Yau $3$-fold.
\end{abstract}

%--------------------Introduction
\section{Introduction}
\label{sect:intro}

Special Lagrangian submanifolds were introduced by Harvey--Lawson \cite{HL82} as an important class of calibrated submanifolds.
They are area minimizers within their homology classes, and a classical problem in geometric analysis is whether a given homology class can be represented by a special Lagrangian.

However, very few constructions of special Lagrangians are known.
One method is to evolve a Lagrangian submanifold along the mean curvature flow, where the limit is a special Lagrangian if the flow converges.
Such machinery was first introduced by Smoczyk \cite{Smo00}.
In general, the flow may not converge and can develop finite time singularities \cite{Nev07}.
Thomas and Yau \cite{TY02} further studied the connection between stability and the mean curvature flow.
In particular, there exists at most one special Lagrangian representative in each Hamiltonian isotopy class.

From a modern point of view, the mean curvature flow is used to study the Harder--Narasimhan property of the Fukaya category as suggested by Joyce \cite{Joy15}.
Many stability conditions have been constructed on the derived category of coherent sheaves on projective varieties. However, little is known about stability conditions on Fukaya categories beyond dimension one. 
It is a folklore conjecture that the stability condition on a Fukaya category is related to the complex structures on the underlying manifold.
In particular, the stable objects would be given by special Lagrangians, and the Harder--Narasimhan property would imply the following statement at the homological level:

\begin{displayquote}
\it The homology class of a Lagrangian submanifold is a sum of classes represented by special Lagrangian submanifolds.
\end{displayquote}

In this paper, we prove the above statement for a dense subset of K\"ahler classes in the K\"ahler cone of an arbitrary K3 surface (not necessarily algebraic). Notice that the above statement is not always true: a counterexample can be found in \cite{Wol05}.

\begin{thm}[$\subseteq$ Theorem~\ref{thm:denseInKahlerCone}]
\label{thm:denseInKahlerCone_Intro}
For every K3 surface $X$, its K\"ahler cone $\cK_X$ contains a dense subset $\mathcal{S}_X$ such that every Lagrangian class decomposes as a sum of classes represented by special Lagrangian submanifolds with respect to a Ricci-flat K\"ahler form $\omega$ parametrized by $\mathcal{S}_X$.
Moreover, the set $\mathcal{S}_X$ contains all rational K\"ahler classes, so this property holds in particular for polarized K3 surfaces.
\end{thm}

Another relevant aspect related to special Lagrangians comes from the Strominger--Yau--Zaslow conjecture \cite{SYZ96}. It predicts that a Calabi--Yau manifold $X$ near large complex structure limit admits a special Lagrangian torus fibration. 
%Moreover, the mirror $\check{X}$ is simply the dual torus fibration of the fibration on $X$.  %However, very few explicit examples of special Lagrangian fibrations have been found.
In the case of K3 surfaces this is already known, and the proof relies on a standard trick involving hyperk\"ahler rotation \cite{GW00} (see also \cite{OO21}*{Section~4.4}). Our contribution is to study the distribution of the K\"ahler classes in the K\"ahler cone of a K3 surface $X$ with respect to which $X$ admits a special Lagrangian fibration.

%Another important aspect related to special Lagrangians comes from the Strominger--Yau--Zaslow conjecture \cite{SYZ96}, which predicts that Calabi--Yau manifolds admit special Lagrangian fibrations near the large complex structure limits. \luca{(The fact that CY manifolds admit sLag fibrations near large complex structure limits is a consequence of the SYZ conjecture, or is equivalent to it?)} This was proved in \cite{OO21}, \luca{Did they prove it only for K3 surfaces, or in general for CYs?} and serves as a guide in the study of mirror symmetry, and many of its implications have been verified.
%However, the conjecture \luca{which conjecture? SYZ? Or the possible consequence of CY manifolds that admit sLag fibrations near large complex structure limits?} itself is only known for complex tori and hyperk\"ahler rotation of hyperk\"ahler manifolds with holomorphic Lagrangian fibrations.
%Here we provide more examples in the case of K3 surfaces as well as its conclusion on log Calabi--Yau $3$-folds.

\begin{thm}[$\subseteq$ Theorem~\ref{thm:SYZ}]
For every K3 surface $X$, its K\"ahler cone $\cK_X$ contains a dense subset $\mathcal{T}_X$ such that every Ricci-flat K\"ahler form $\omega$ parametrized by $\mathcal{T}_X$ admits a special Lagrangian fibration.
Moreover, the set $\mathcal{T}_X$ contains all rational K\"ahler classes, so this property holds in particular for polarized K3 surfaces.
\end{thm}

In the case of log Calabi--Yau $3$-folds, the above theorem in the algebraic case together with \cite{CJL21} imply the following result.

\begin{cor}[= Corollary~\ref{cor:3ToriInCY3}]
Let $Y$ be a Fano $3$-fold and $D\in |-K_Y|$ be a smooth anti-canonical divisor. Then the log Calabi--Yau $3$-fold $X=Y\backslash D$ contains infinitely many special Lagrangian $3$-tori.
\end{cor}

This paper is organized as follows. In Section~\ref{sect:lagK3}, we review some basic definitions and techniques which are necessary to our work.
We also provide a sufficient condition for a Lagrangian class to be decomposed into special Lagrangian classes.
In Section~\ref{sect:decompCriteria}, we give criteria for the decomposability in terms of lattice theory and then prove the density property as a consequence.
The SYZ conjecture for algebraic K3 surfaces is discussed in the end of the same section.

%----------Acknowledgements
\subsection*{Acknowledgements}
The authors would like to thank Yu-Wei Fan, Fran\c cois Greer, Eyal Markman, Yuji Odaka, and Colleen Robles for helpful discussions. The second author is supported by the Simons Collaboration grant \#635846 and the NSF grant DMS \#2204109.

%--------------------Lagrangian classes on K3 surfaces
\section{Lagrangian classes on K3 surfaces}
\label{sect:lagK3}

In this section, we start with preliminaries about basic definitions and core machinery necessary to our work.
The remaining part is devoted to developing a sufficient condition for a Lagrangian class to be a sum of classes of special Lagrangian submanifolds.
Several parts of this section should be well-known to experts.

%----------Lagrangians on Calabi--Yau manifolds
\subsection{Lagrangians on Calabi--Yau manifolds}
\label{subsect:defineLagrangian}

Let $X$ be a Calabi--Yau manifold of complex dimension $n$ and fix a Ricci-flat K\"ahler form $\omega$ on $X$.
An \emph{immersed Lagrangian submanifold} of $X$ with respect to $\omega$ is an immersion
\[
\xymatrix{
    \iota\colon L\ar[r] & X
}
\]
of a connected and oriented $n$-dimensional real manifold $L$ such that $\iota^*\omega=0$.
Throughout the paper, we may call such an immersion briefly a \emph{Lagrangian immersion}, or simply a \emph{Lagrangian}.
When there is no confusion, we may directly call $L$ a Lagrangian without specifying the immersion.

The K\"ahler form $\omega$ determines a Riemannian metric on $X$ and thus on $L$, which then induces a Riemannian volume form $dV_L$.
On the other hand, the covariantly constant holomorphic top form $\Omega$ on $X$ restricts as a nowhere vanishing top form $\iota^*\Omega$ on $L$.
So we can write
\[
    \iota^*\Omega = e^{i\phi}dV_L
\]
where $\phi\colon L\rightarrow \mathbb{R}/\mathbb{Z}$ is a differentiable function on $L$ called the \emph{phase}.
In the case that $\iota$ is an embedding and the phase $\phi$ is a constant function, we call $L$ a \emph{special Lagrangian submanifold}, or briefly a \emph{special Lagrangian}.
Note that this definition is independent of the choice of $\Omega$. It is well-known that special Lagrangians are volume minimizers among Lagrangians \cite{HL82}.

We call a homology class $\gamma\in H_n(X,\bZ)$ a \emph{Lagrangian class} if it can be represented by a Lagrangian immersion, and we call it a \emph{special Lagrangian class} if it is a sum of classes represented by special Lagrangians with positive coefficients.
Note that, under this definition, a representative of a special Lagrangian class may be a union of special Lagrangian submanifolds.
Using Poincar\'e duality, we consider a (special) Lagrangian class as a class in $H^n(X,\bZ)$ throughout the paper.

%----------The trick of hyperk\"ahler rotation
\subsection{The hyperk\"ahler rotation trick}
\label{subsect:hyperkahlerRotation}

%Let us briefly review the core concept behind our work.
A K\"{a}hler manifold $(X, \omega)$ of $\mbox{dim}_{\mathbb{C}}X=2n$ is \emph{hyperk\"{a}hler} if its holonomy group is in $\Sp(n)$. From the celebrated Calabi conjecture \cite{Y78}, every holomorphic symplectic manifold is hyperk\"ahler. 

From the $\Sp(n)$-holonomy, there exists a covariantly costant holomorphic symplectic $2$-form $\theta$ of norm $\sqrt{2}$ unique up to $S^1$-scaling.
A \emph{holomorphic Lagrangian} $L\subseteq X$ is a half-dimensional complex submanifold which satisfies $\theta|_L=0$. 
Given any
\[
    \zeta\in\{z\in\bC:|z|=1\}\cong S^1
\]
we denote by $X_{\zeta}$ the K\"ahler manifold with the same underlying space as $X$ and equipped with K\"ahler form and holomorphic symplectic $2$-form given~by
\begin{equation}
\label{eqn:HKrot}
\begin{aligned}
    \omega_{\zeta}&\colonequals\mathrm{Re}(\zeta\theta)\\
    \theta_{\zeta}&\colonequals\omega-i\mathrm{Im}(\zeta\theta).
\end{aligned}
\end{equation}
The following lemma is known as the hyperk\"ahler rotation trick.
It can be derived from \eqref{eqn:HKrot} and \cite{HL82}*{III~Corollary~1.11}.

\begin{lemma}
\label{lemma:holo-HK-Lag}
Let $X$ be a hyperk\"ahler manifold and let $L\subseteq X$ be a half-dimensional real submanifold.
\begin{enumerate}[label=\textup{(\arabic*)}]
    \item\label{holo-To-Lag}
    If $L$ is a holomorphic Lagrangian in $X$, then the same underlying space of $L$ defines a special Lagrangian in $X_{\zeta}$ for any $\zeta\in S^1$.
    \item\label{lag-To-Holo}
    If $L$ is a Lagrangian in $X$ such that $\mathrm{Im}(\zeta\theta)|_L=0$ for some $\zeta\in S^1$, then $L$ is a holomorphic Lagrangian in $X_{\zeta}$.
\end{enumerate}
In particular, if $\dim_{\mathbb{C}}X=2$, then $L\subseteq X$ is a holomorphic Lagrangian if and only if it is a special Lagrangian up to a hyperk\"ahler rotation.
\end{lemma}

\begin{lemma}
\label{lemma:spLagK3}
Let $X$ be a K3 surface equipped with a Ricci-flat K\"ahler form $\omega$ and a holomorphic $2$-form $\theta$.
\begin{enumerate}[label=\textup{(\alph*)}]
    \item\label{holo-To-Lag_zeta}
    If $C$ is a holomorphic curve in $X_\zeta$ for some $\zeta\in S^1$, then $C$ is a special Lagrangian in $X$.
    \item\label{spLag-2}
    If $\gamma\in H^2(X,\bZ)$ is represented by an irreducible special Lagrangian submanifold, then $\gamma^2\geq -2$.
\end{enumerate}
\end{lemma}
\begin{proof}
Part~\ref{holo-To-Lag_zeta} is a special case of Lemma~\ref{lemma:holo-HK-Lag}~\ref{holo-To-Lag}. It can also be derived straightforwardly from \eqref{eqn:HKrot}.
Part~\ref{spLag-2} follows from Lemma~\ref{lemma:holo-HK-Lag} and \cite{BHPV04}*{VIII~Proposition~3.7~(ii)}.
\end{proof}

\begin{rmk}
In higher dimension, there exist special Lagrangians which never become holomorphic Lagrangians after any hyperk\"ahler rotation. See \cite{Hat14} for such examples.
\end{rmk}

%----------Decomposing classes of type (1,1)
\subsection{Decomposing classes of type (1,1)}
\label{subsect:holoK3}

Here we prove a sufficient condition for classes of Hodge type $(1,1)$ to be decomposed into classes of smooth holomorphic curves on a K3 surface.
Later on we will combine it with the hyperk\"ahler rotation to develop a similar condition for Lagrangian classes.

\begin{lemma}
\label{lemma:effectiveClasses}
Let $X$ be a K3 surface (notice that there is no projectivity assumption on $X$) equipped with a holomorphic $2$-form $\theta$.
Assume that $\gamma\in H^2(X,\bZ)$ satisfies $\int_{\gamma}\theta=0$ and $\gamma^2\geq -2$.
Then $\gamma$ or $-\gamma$ can be written as
\[
    n_1[C_1] + ... + n_k[C_k],\quad n_i\geq 0,
\]
where each $C_i\subseteq X$ is a smooth holomorphic curve.
\end{lemma}

\begin{proof}
The condition $\int_{\gamma}\theta=0$ means that
\[
    \gamma\in H^{1,1}(X,\bC)\cap H^2(X,\bZ).
\]
By \cite{BHPV04}*{VIII~Propositions~3.7~(i)}, either $\gamma$ or $-\gamma$ is effective.
Assume that $\gamma$ is effective.
As a consequence of \cite{BHPV04}*{VIII~Propositions~3.8}, we can decompose $\gamma$ into classes of irreducible curves as
\[
    \gamma = \sum_in_i[C_i]+ \sum_jm_j[E_j] + \sum_k\ell_k[D_k],
\]
where $n_i$, $m_j$, $\ell$ are non-negative integers, $C_i^2=-2$, $E_j^2=0$, and $D_k^2>0$.
The curves $C_i$ are $(-2)$-curves and thus smooth.
The linear system $|E_j|$ has no base-point, so we can apply Bertini's theorem \cite{GH94}*{p.137} to choose a smooth irreducible representative of the class of $E_j$.

Suppose that $\ell_k\neq0$ for some $k$.
In this case, $X$ is algebraic by \cite{BHPV04}*{IV~Theorem~6.2}.
Evidently, $D_k$ has no fix component, so we can apply \cite{Sai74}*{Corollary~3.2} to conclude that $|D_k|$ has no base-point.
Again by Bertini's theorem we can choose a smooth irreducible representative of the class of $D_k$.
This proves our claim for $\gamma$.
In the case that $-\gamma$ is effective, an analogous argument yields our claim for $-\gamma$ instead.
\end{proof}

\begin{lemma}
\label{lemma:lag(1,1)}
Let $X$ be a K3 surface equipped with a Ricci-flat K\"ahler form $\omega$ and a holomorphic $2$-form $\theta$.
Assume that $\gamma\in H^2(X,\bZ)$ is a Lagrangian class which satisfies $\int_\gamma\theta=0$ and $\gamma^2\geq-2$. Then $\gamma=0$.
\end{lemma}
\begin{proof}
By Lemma~\ref{lemma:effectiveClasses}, we can write $\gamma = \sum_{i=1}^kn_i[C_i]$ where $n_i\geq 0$ for all $i$ and each $C_i\subseteq X$ is a holomorphic curve.
As $[C_i]\cdot[\omega]>0$ for all $i$, the condition $\gamma\cdot[\omega]=0$ implies that $n_i=0$ for all $i$ and thus $\gamma=0$.
\end{proof}

%----------Decomposing Lagrangian classes
\subsection{Decomposing Lagrangian classes}
\label{subsect:sufficientCondition}

Here we introduce the numerical condition for a Lagrangian class to be decomposed into classes of special Lagrangian submanifolds.
Before showing the sufficiency of the condition,
let us prove a lemma.

\begin{lemma}
\label{lemma:producingClassType(1,1)}
Let $X$ be a K3 surface equipped with a Ricci-flat K\"ahler form $\omega$ and a holomorphic $2$-form $\theta$.
Suppose that $\gamma\in H^2(X,\bZ)$ is a Lagrangian class such that $c\colonequals\int_{\gamma}\theta\neq0$.
Then $\int_{\gamma}\theta_\zeta = 0$ if and only if $\zeta^2=\bar{c}/c$.
\end{lemma}
\begin{proof}
Using the fact that $\int_\gamma\omega=0$, we obtain
\begin{gather*}
    0
    =\int_\gamma\theta_\zeta
    = \int_\gamma(\omega-i\mathrm{Im}(\zeta\theta))
    = -i\int_\gamma\mathrm{Im}(\zeta\theta)\\
    = -\frac{1}{2}\int_\gamma(\zeta\theta-\overline{\zeta\theta})
    = -\frac{1}{2}\left(
    \zeta\int_\gamma\theta
    - \frac{1}{\zeta}\int_\gamma\bar{\theta}
    \right)
    = -\frac{1}{2}\left(
    \zeta c - \frac{1}{\zeta}\bar{c}
    \right).
\end{gather*}
This is equivalent to
\[
    \zeta c - \frac{1}{\zeta}\bar{c} = 0,
\]
which holds if and only if $\zeta^2 = \bar{c}/c$.
\end{proof}

\begin{prop}
\label{prop:LagToSpLags-2}
Let $X$ be a K3 surface equipped with a Ricci-flat K\"ahler form $\omega$.
If $\gamma\in H^2(X,\mathbb{Z})$ is a Lagrangian class satisfying $\gamma^2\geq -2$,
then
\[
    \gamma = \sum_{i=1}^k \alpha_i,
\]
where each $\alpha_i$ is represented by a special Lagrangian submanifold of the same phase. In other words, $\gamma$ is a special Lagrangian class.
\end{prop}
\begin{proof}
If $\int_\gamma\theta=0$, then the conclusion follows from Lemma~\ref{lemma:lag(1,1)}.
Assume that $\int_\gamma\theta\neq 0$.
Then there exists $\zeta$ such that $\int_\gamma\theta_\zeta = 0$ by Lemma~\ref{lemma:producingClassType(1,1)}.
According to Lemma~\ref{lemma:effectiveClasses}, either $\gamma$ or $-\gamma$ can be written as a sum
\begin{equation}
\label{gammaor-gammaequaltosum}
\sum_{i=1}^k\alpha_i,
\end{equation}
where each $\alpha_i$ is represented by a smooth holomorphic curves on $X_{\zeta}$.
Here we allow the terms in the sum to occur multiple times.
Then by Lemma~\ref{lemma:spLagK3}~\ref{holo-To-Lag_zeta}, each $\alpha_i$ is represented by a special Lagrangian in $X$. Therefore, if $\gamma$ is equal to the sum in \eqref{gammaor-gammaequaltosum}, then we are done. Otherwise, if $-\gamma$ is equal to \eqref{gammaor-gammaequaltosum}, then write
\[
    \gamma=\sum_{i=1}^k(-\alpha_i)
\]
and we have that the $-\alpha_i$ are holomorphic for $X_{-\zeta}$.
Applying Lemma~\ref{lemma:spLagK3}~\ref{holo-To-Lag_zeta} again, the classes $-\alpha_i$ are represented by special Lagrangians.
\end{proof}

%--------------------Criteria for the decomposability
\section{Criteria for the decomposability}
\label{sect:decompCriteria}

In this section, we give criteria which determine whether a Lagrangian class on a K\"ahler K3 surface is special, i.e. a sum of classes represented by special Lagrangian submanifolds, or not.
The criteria are developed upon the lattice structure of the second cohomology group.
They allow us to describe the locus in the moduli space about the K\"ahler K3 surfaces on which every Lagrangian class is special.
In the end of the section, we discuss the existence of a special Lagrangian fibration on algebraic K3 surfaces.

%----------Sublattice of Lagrangian classes
\subsection{Sublattice of Lagrangian classes}
\label{subsect:lagSublat}

Let $X$ be a K3 surface. As a reminder, recall that the cohomology group $H^2(X,\bZ)$ equipped with the intersection product is isometric to the lattice
\[
    \Lambda_{\Kthree}\colonequals
    U^{\oplus 3}\oplus E_8^{\oplus 2},
\]
where $U$ is the hyperbolic plane and $E_8$ is the unique negative definite unimodular even lattice of rank $8$.

Now equip $X$ with a Ricci-flat K\"ahler form $\omega$ and covariantly constant holomorphic volume form $\theta$.
We denote by $[\omega]^{\perp}$ the orthogonal complement of $[\omega]$ in
\[
    H^2(X,\bR) = H^2(X,\bZ)\otimes_{\bZ}\bR.
\]
If a class $\gamma\in H^2(X,\bZ)$ is represented by a Lagrangian immersion $\iota\colon L\to X$, then the relation $\iota^*\omega=0$ implies that
\[
    \gamma\cdot[\omega]
    = \int_{L}\iota^*\omega
    = 0.
\]
The converse is true by \cite{SW01}*{Corollary~2.4}.
As a consequence, a class $\gamma\in H^2(X,\bZ)$ is Lagrangian if and only if $\gamma$ belongs to the sublattice
\begin{equation}
\label{eqn:Lag(KahlerK3)}
    \Lag(X,\omega)
    \colonequals
    [\omega]^\perp\cap H^2(X,\bZ),
\end{equation}
which we call a \emph{Lagrangian lattice}.
Due to Yau's celebrated theorem \cite{Y78}, there is a unique Ricci-flat metric in each K\"ahler class. We will no longer distinguish a K\"ahler class and a Ricci-flat K\"ahler form later on.
We further define $\SLag(X,\omega)$ as the sublattice of $\Lag(X,\omega)$ generated by the classes of special Lagrangian submanifolds. 
In other words, $\SLag(X,\omega)$ consists of special Lagrangian classes.

Our main question is about whether $\Lag(X,\omega)$ and $\SLag(X,\omega)$ coincide or not.
First of all, we have the following numerical characterization for $\SLag(X,\omega)$.

\begin{prop}
\label{prop:numericalSpLag}
$
    \SLag(X,\omega)=\langle
        \delta\in\Lag(X,\omega):
        \delta^2\geq-2
    \rangle.
$
\end{prop}

\begin{proof}
The containment $(\supseteq)$ follows from Proposition~\ref{prop:LagToSpLags-2}. The other containment follows from Lemma~\ref{lemma:spLagK3}~\ref{spLag-2}.
\end{proof}

%----------Characterization of Lagrangian lattices
\subsection{Characterization of Lagrangian lattices}
\label{subsect:charLagSublat}

Which sublattice $E\subseteq\Lambda_{\Kthree}$ appears as a Lagrangian lattice?
Suppose that $E=\Lag(X,\omega)$ for some K\"ahler K3 surface $(X,\omega)$. Then clearly $E$ is proper and saturated (by proper we mean properly contained in $\Lambda_{\Kthree}$, and it could possibly be zero). Moreover, the subspace $(E^\perp)_\bR\subseteq(\Lambda_{\Kthree})_\bR$ contains a vector with positive self-intersection, namely, the K\"ahler class $[\omega]$, which implies that $E^\perp$ itself contains a vector with positive self-intersection.
In the following, we will show that these conditions are sufficient for $E$ to be a Lagrangian sublattice.

\begin{lemma}
\label{lemma:realizableOrthComp}
Let $\Lambda$ be a non-degenerate lattice and let $E\subsetneq\Lambda$ be a proper saturated sublattice. Suppose that the orthogonal complement $E^\perp\subseteq\Lambda$ contains a vector with positive self-intersection.
Then there exists $v\in\Lambda_\bR\colonequals\Lambda\otimes\bR$ such that $v^2>0$ and $E = v^\perp\cap\Lambda$.
\end{lemma}
\begin{proof}
By hypothesis, there exists $x\in E^\perp$ such that $x^2>0$.
Suppose that $E^\perp$ is spanned by the basis
$
    \{e_1,\dots,e_m\}\subseteq\Lambda\setminus\{0\}.
$
Let $\{c_1,\dots,c_m\}\subseteq\bR$ be a set of real numbers ordered in such a way that the field extensions
\[
\mathbb{Q}\subseteq\mathbb{Q}(c_1)\subseteq\ldots\subseteq\mathbb{Q}(c_1,\ldots,c_m)
\]
%\[
%    \bQ(c_1,\dots,c_{i-1},c_{i})
%    = \bQ(c_1,\dots,c_{i-1})(c_{i})
%\]
%is a transcendental extension of $\bQ(c_1,\dots,c_{i-1})$ for $i=1,\dots,m$.
are all transcendental extensions.
(For example, we can choose $c_i=2^{\pi^{i}}$.)~Let
\[
    y = \sum_{i=1}^mc_ie_i,
\]
and define
\[
    v\colonequals x+\epsilon y,
    \quad\epsilon\in\bQ\setminus\{0\}.
\]
Then $v^2>0$ if $|\epsilon|\ll 1$.

The inclusion $E\subseteq v^\perp\cap\Lambda$ is clear. We claim that $E = v^\perp\cap\Lambda$.
Assume, to the contrary, that there exists $w\in v^\perp\cap\Lambda$ such that $w\notin E$.
The condition $w\in v^\perp$ implies that
\[
    0 = v\cdot w
    = (x+\epsilon y)\cdot w
    = x\cdot w + \epsilon(y\cdot w),
\]
which is equivalent to
\begin{equation}
\label{eqn:linearCombTranscend}
    -\frac{1}{\epsilon}(x\cdot w)
    = y\cdot w
    = \sum_{i=1}^m c_i(e_i\cdot w).
\end{equation}
Note that $(E^\perp)^\perp = E$ since $\Lambda$ is non-degenerate and $E$ is saturated.
If $(e_i\cdot w) = 0$ for all $i$, then $w\in (E^\perp)^\perp = E$ which contradicts to the assumption that $w\notin E$.
Therefore, \eqref{eqn:linearCombTranscend} gives a nontrivial algebraic relation among $c_1$, ..., $c_m$ over $\bQ$,
which is impossible.
This completes the proof.
\end{proof}

Before proceeding to the next result, let us briefly review some background material. Our main reference is \cite{Huy16}*{Chapter~6}.
Let $\Lambda\colonequals\Lambda_{\Kthree}$.
Recall that the period domain of K3 surfaces is defined as
\[
    \cD\colonequals\{
        [x]\in\bP(\Lambda_\bC):
        x^2=0,\; x\cdot\bar{x}>0
    \}.
\]
By the Torelli theorem, for every period point $[x]\in\cD$, there exists a K3 surface $X$ together with an isometry
\[\xymatrix{
    \varphi\colon H^2(X,\bZ)\ar[r]^-\sim & \Lambda_{\Kthree},
}\]
which satisfies $\varphi(H^{2,0}(X,\bC)) = \bC x$ after extending the scalars to $\bC$.
It is a general fact that, given any analytic submanifold $B\subseteq\cD$, there exists a dense subset $NL(B)\subsetneq B$ consisting of countably many divisors, called the \emph{Noether--Lefschetz locus}, such that the Picard numbers of the K3 surfaces parametrized by $B$ jump on $NL(B)$ and retain minimum on the complement.

\begin{prop}
\label{prop:realizableLag}
Let $E\subsetneq\Lambda_{\Kthree}$ be a sublattice.
Then there exists a K\"ahler K3 surface $(X,\omega)$ such that $\Lag(X,\omega)\cong E$ if and only if $E$ is proper, saturated, and such that $E^\perp\subseteq\Lambda_{\Kthree}$ contains a vector with positive self-intersection.
\end{prop}
\begin{proof}
The ``only if'' part is clear from the description before Lemma~\ref{lemma:realizableOrthComp}.
To prove the ``if'' part, let us denote $\Lambda\colonequals\Lambda_{\Kthree}$.
By Lemma~\ref{lemma:realizableOrthComp}, there exists $v\in\Lambda_\bR$ such that $v^2>0$ and $E=v^\perp\cap\Lambda$.
If we can find $x\in\cD$ such that either $v$ or $-v$ corresponds to a K\"ahler class $[\omega]$ on a K3 surface $X$ parametrized by $x$, then
\[
    \Lag(X,\omega)
    \cong[\omega]^\perp\cap\Lambda
    \cong E
\]
as desired.
We divide the proof into two cases.

First assume that $rv\in\Lambda$ for some $r\in\bR$. We can further assume that $rv$ is primitive in $\Lambda$ and let $d\colonequals(rv)^2$.
Note that $d$ is a positive even integer since $v^2>0$ and $\Lambda$ is even.
Then the locus
\[
    \cD_d\colonequals\bP(v^\perp)\cap\cD
\]
parametrizes polarized K3 surfaces of degree $d$,
and a very general point $x\in\cD_d$ corresponds to a K3 surface $X$ of Picard number one.
On such $X$, either $rv$ or $-rv$ is the class of an ample line bundle and thus is a K\"ahler class, so we finished this case.

Now assume that $rv\notin\Lambda$ for any $r\in\bR\setminus\{0\}$. 
A K3 surface $X$ parametrized by a very general point of the locus
\[
    \bP(v^\perp)\cap\cD
\]
has Picard number zero.
For such $X$, the K\"ahler cone and the positive cone coincide \cite{Huy16}*{Chapter~8, Corollary~5.3}.
In particular, $v^2>0$ implies that $v$ is a K\"ahler class. This completes the proof.
\end{proof}

\begin{cor}
\label{cor:lagE8}
There exists a K\"ahler K3 surface $(X,\omega)$ such that $\Lag(X,\omega)$ is isomorphic to the $E_8$ lattice.
\end{cor}
\begin{proof}
Consider $E_8$ as one of the $E_8$ copies in $\Lambda_{\Kthree}=U^{\oplus 3}\oplus E_8^{\oplus 2}$.
Then we have $E_8^\perp\cong U^{\oplus 3}\oplus E_8$, where one can easily produce a vector with positive self-intersection from the copies of $U$'s.
In particular, the existence is confirmed by Proposition~\ref{prop:realizableLag}.
\end{proof}

%----------The lattice-theoretic criteria
\subsection{The lattice-theoretic criteria}
\label{subsect:latticeCriteria}

Here we introduce our main result, namely, the criteria for determining whether $\SLag(X,\omega)$ is the whole $\Lag(X,\omega)$ or not. The criteria are established upon the following observation on lattices.

\begin{lemma}
\label{lemma:positiveOrNot}
Let $L$ be an arbitrary lattice.
\begin{enumerate}[label=\textup{(\roman*)}]
    \item\label{withPositive}
    If $L$ contains a vector with positive self-intersection, then every $\gamma\in L$ decomposes as $\gamma=\alpha+\beta$ with $\alpha^2>0$ and $\beta^2>0$.
    \item\label{nonorthogonalIsotropic}
    If $L$ contains an isotropic vector $\delta$ such that $\delta\cdot\alpha\neq0$ for some $\alpha\in L$, then $L$ contains a vector of positive self-intersection.
    \item\label{withoutPositive}
    If $L$ contains no vector with positive self-intersection,
    then there is an orthogonal decomposition $L\cong\bZ\oplus\cdots\oplus\bZ\oplus N$
    where each copy of $\bZ$ is spanned by an isotropic vector and $N$ is negative definite.
\end{enumerate}
\end{lemma}
\begin{proof}
Assume that there exists $x\in L$ such that $x^2>0$.
Let $\gamma\in L$ be an arbitrary vector and let $m$ be an integer.
Define $\alpha=mx$ and $\beta=\gamma-mx$.
Then $\gamma=\alpha+\beta$ and $\alpha^2=m^2(x^2)>0$.
Moreover, we have
\[
    \beta^2
    = (\gamma-mx)^2
    =\gamma^2-2m(\gamma\cdot x)+m^2(x^2),
\]
which becomes positive by picking $m\gg0$. Hence \ref{withPositive} is proved.

To prove \ref{nonorthogonalIsotropic}, let $m$ be any integer and consider the vector $m\delta+\alpha$. Since $\delta^2=0$, we have
\[
    (m\delta+\alpha)^2
    = 2m(\alpha\cdot\delta) + \alpha^2.
\]
If $\delta\cdot\alpha>0$, then $m\delta+\alpha$ has positive self-intersection as $m\gg0$.
If $\delta\cdot\alpha<0$, we can choose $m\ll0$ instead.

If $L$ contains no vector with positive self-intersection, then \ref{nonorthogonalIsotropic} implies that every isotropic vector in $\Lag(X,\omega)$ is orthogonal to every other vector.
Therefore, we have the claimed decomposition, which proves \ref{withoutPositive}.
\end{proof}

\begin{thm}
\label{thm:LagToSpLag}
Let $X$ be a K3 surface equipped with a K\"ahler form $\omega$.
\begin{enumerate}[label=\textup{(\arabic*)}]
    \item\label{withPositiveLag}
    If $\Lag(X,\omega)$ contains a vector with positive self-intersection, then
    \begin{equation}
    \label{eqn:Lag=SLag}
        \Lag(X,\omega)
        =\SLag(X,\omega)
        =\langle\delta\in\Lag(X,\omega):
            \delta^2\geq-2
        \rangle.
    \end{equation}
    \item\label{withoutPositiveLag}
    If $\Lag(X,\omega)$ contains no vector with positive self-intersection, then there is an orthogonal decomposition
    \[
        \Lag(X,\omega)
        \cong\bZ\oplus\cdots\oplus\bZ\oplus N,
    \]
    where each copy of $\bZ$ is spanned by an isotropic vector and $N$ is negative definite.
    In this case, \eqref{eqn:Lag=SLag} holds if and only if $N\subseteq\SLag(X,\omega)$.
\end{enumerate}
\end{thm}
\begin{proof}
Assume that the lattice $\Lag(X,\omega)$ contains a vector with positive self-intersection. Then Lemma~\ref{lemma:positiveOrNot}~\ref{withPositive} implies that every $\gamma\in\Lag(X,\omega)$ can be written as $\gamma=\alpha+\beta$ where $\alpha,\beta\in\Lag(X,\omega)$ satisfy $\alpha^2>0$ and $\beta^2>0$.
Together with Proposition~\ref{prop:numericalSpLag}, we obtain
\[
    \Lag(X,\omega)
    \subseteq\langle
        \delta\in\Lag(X,\omega):
        \delta^2\geq-2
    \rangle
    =\SLag(X,\omega)
    \subseteq\Lag(X,\omega),
\]
which implies that the inclusions are actually equalities. This proves \ref{withPositiveLag}.

In condition~\ref{withoutPositiveLag}, the orthogonal decomposition for $\Lag(X,\omega)$ follows immediately from Lemma~\ref{lemma:positiveOrNot}~\ref{withoutPositive}.
Note that the isotropic part is contained in $\SLag(X,\omega)$ due to Proposition~\ref{prop:numericalSpLag}.
As a consequence, $\Lag(X,\omega)=\SLag(X,\omega)$ if and only if $\Lag(X,\omega)\subseteq\SLag(X,\omega)$, which holds if and only if $N\subseteq\SLag(X,\omega)$.
\end{proof}

\begin{rmk}
In condition~\ref{withoutPositiveLag} of Theorem~\ref{thm:LagToSpLag}, the inclusion 
\begin{equation}
\label{eqn:negativeSLag}
    N\subseteq\SLag(X,\omega)
    =\langle\delta\in\Lag(X,\omega):
            \delta^2\geq-2
    \rangle
\end{equation}
may hold or not.
For instance, in \cite{Wol05} Wolfson constructed an example such that $\Lag(X,\omega)\cong N\cong\langle-4\rangle$ which cannot be contained in $\SLag(X,\omega)$.
In particular, the equality $\Lag(X,\omega)=\SLag(X,\omega)$ does not hold for all possible complex structures with $\omega$ as a K\"ahler class.
On the other hand, many sublattices of the K3 lattice $\Lambda_{\Kthree}$ appear as Lagrangian sublattices according to Proposition~\ref{prop:realizableLag}.
For example, the situation $\Lag(X,\omega)\cong E_8$ can occur by Corollary~\ref{cor:lagE8}.
In this case, $N\cong E_8$ is generated by $(-2)$-vectors, so \eqref{eqn:negativeSLag} holds, and thus $\Lag(X,\omega)=\SLag(X,\omega)$.
\end{rmk}

%Immediate cases
Recall that a K3 surface $X$ is algebraic if and only if it admits a K\"ahler form $\omega$ which is \emph{rational}, i.e. the class $[\omega]$ belongs to $H^2(X,\bQ)$.
In this case, we have:

\begin{cor}
\label{cor:projK3}
Assume that $X$ is a K3 surface equipped with a rational K\"ahler form $\omega$.
Then $\Lag(X,\omega) = \SLag(X,\omega)$.
\end{cor}
\begin{proof}
The lattice $\Lag(X,\omega)$ has signature $(2,19)$,
so the conclusion follows immediately from Theorem~\ref{thm:LagToSpLag}~\ref{withPositiveLag}.
\end{proof}

\begin{cor}
\label{cor:rationalPeriodK3}
Assume that $X$ is a K3 surface such that
\begin{equation}
\label{eqn:rationalPeriod}
    (H^{2,0}(X)\oplus H^{0,2}(X))\cap H^2(X,\bQ)\neq\{0\}.
\end{equation}
Let $\omega$ be any K\"ahler form $\omega$ on $X$.
Then $\Lag(X,\omega) = \SLag(X,\omega)$.
\end{cor}
\begin{proof}
Let $v$ be a nonzero vector in $(H^{2,0}(X)\oplus H^{0,2}(X))\cap H^2(X,\bQ)$. We have that $v\cdot\omega=0$. Moreover, $v^2>0$ because $H^{2,0}(X)\oplus H^{0,2}(X)$ is positive definite. Then the conclusion follows from Theorem~\ref{thm:LagToSpLag}~\ref{withPositiveLag}.
\end{proof}

\begin{rmk}
The K3 surfaces satisfying the assumption of Corollary~\ref{cor:rationalPeriodK3} form a family of $\bR$-dimension $20$ in the moduli of K3 surfaces.
\end{rmk}

%----------The distribution in the K\"ahler cone
\subsection{Distribution of K\"ahler classes giving the decomposition}
\label{subsect:kahlerCone}

Given a K3 surface $X$, for which class $[\omega]$ in the K\"ahler cone \[
    \cK_X\subseteq H^{1,1}(X,\bR)
\]
do we have $\Lag(X,\omega)=\SLag(X,\omega)$?
In the situation where $\Lag(X,\omega)=0$, which does occur by Proposition~\ref{prop:realizableLag}, the sublattice $\SLag(X,\omega)$ must be zero as well, hence the equality holds trivially.
Therefore, we will focus on the case that $\Lag(X,\omega)\neq0$.

The cone $\cC\colonequals\{x\in H^{1,1}(X,\bR):x^2>0\}$
consists of two connected components.
Let $\cC_X$ denote the positive cone, that is, the component which contains $\cK_X$.
We first show that, in fact, a similar density property can be seen in $\cC_X$ already, in a sense made precise in the next proposition.

\begin{prop}
\label{prop:denseInPosCone}
On a K3 surface $X$ not satisfying the hypothesis of Corollary~\ref{cor:rationalPeriodK3}, the classes $w\in\cC_X$ whose orthogonal complement $w^\perp\cap H^2(X,\bZ)$ contains a vector $v$ with $v^2>0$ form a dense subset $\mathcal{P}_X\subseteq\cC_X$ in the analytic topology.
This subset is a union of countably many hyperplane sections and it contains $\cC_X\cap H^{1,1}(X, \bQ)$.
\end{prop}

The key idea behind the proof of the above proposition is the following lemma.

\begin{lemma}
\label{lemma:denseHyperSec}
Let $\cC^c$ be the complement of $\cC$ in $H^{1,1}(X,\bR)$.
Suppose that $\mathcal{A}\subseteq\cC^c\setminus\{0\}$ is a dense subset.
Then the union
\[
    \bigcup_{a\in\mathcal{A}}a^\perp\cap\mathcal{C}_X
\]
is dense in $\cC_X$.
\end{lemma}

\begin{proof}
Let $x\in\cC_X$. Then $x^\perp\cap\cC_X=\emptyset$ by \cite{BHPV04}*{IV~Corollary~7.2}.
In particular, we have $x^\perp\subseteq\cC^c$. If, by contradiction, $x^\perp$ does not intersect $\overline{\cC}^c$, then it would be contained in the boundary of $\cC$, which cannot happen.
So there exists $a\in\overline{\cC}^c$ such that $a\cdot x=0$.
It follows that
\[
    \cC_X\subseteq\bigcup_{a\in\overline{\cC}^c}a^\perp\subseteq\bigcup_{a\in\cC^c\setminus\{0\}}a^\perp.
\]
Hence it is sufficient to show that
\begin{equation}
\label{eqn:denseHyperInHyper}
    \bigcup_{a\in\mathcal{A}}a^\perp
    \subseteq
    \bigcup_{a\in\cC^c\setminus\{0\}}a^\perp
\end{equation}
is a dense subset.
For simplicity of notation, let $V\colonequals H^{1,1}(X,\bR)$. Consider the diagram
\[\xymatrix{
    & \mathcal{U}\colonequals\{(a,z)\in\cC^c\setminus\{0\}\times V : a^\perp\ni z\}\ar[dl]_-{\pi_1}\ar[dr]^-{\pi_2} &\\
    \cC^c\setminus\{0\} && V.
}\]
Then $\pi_1$ is a vector bundle fibered in $19$-dimensional real vector spaces, and the image of $\pi_2$ coincides with the right hand side of \eqref{eqn:denseHyperInHyper}. In particular, $\pi_1$ is an open map. 
Because $\mathcal{A}$ is dense in $\cC^c\setminus\{0\}$, $\pi_1^{-1}(\mathcal{A})$ is dense in $\mathcal{U}$ and thus
\[
    \bigcup_{a\in\mathcal{A}}a^\perp
    = \pi_2(\pi_1^{-1}(\mathcal{A}))
\]
is dense in the image of $\pi_2$, which completes the proof.
\end{proof}

\begin{proof}[Proof of Proposition~\ref{prop:denseInPosCone}]
The goal is to write $\mathcal{P}_X$ as in the statement of Lemma~\ref{lemma:denseHyperSec} for some dense $\mathcal{A}\subseteq\cC^c\setminus\{0\}$. Define
\[
    \cQ^+\colonequals\{x\in H^2(X,\bQ):x^2>0\}.
\]
First note that
\[
	\mathcal{P}_X = \bigcup_{x\in\cQ^+}x^\perp\cap\cC_X.
\]
Indeed, if $w\in\cC_X$ admits the existence of $v\in w^\perp\cap H^2(X, \bZ)$ with $v^2>0$, then $v\in\cQ^+$ and $w\in v^\perp\cap\cC_X$. Conversely, for every $w\in x^\perp\cap\cC_X$, we can multiply $x$ by some integer $n$ to get $v\colonequals nx\in H^2(X, \bZ)$. Then $v^2>0$, and $w\in x^\perp$ implies that $v=nx\in w^\perp$.

Consider the orthogonal decomposition
\[
    H^2(X,\bR)\cong
    H^{1,1}(X,\bR)\oplus(H^{2,0}(X)\oplus H^{0,2}(X))_\bR
\]
together with the induced projection
\[\xymatrix{
    \pi\colon H^2(X,\bR)\ar[r] & H^{1,1}(X,\bR).
}\]
It is easy to verify that $x^\perp\cap\cC_X=\pi(x)^\perp\cap\cC_X$ for all $x\in H^2(X,\bR)$, so
\[
	\mathcal{P}_X = \bigcup_{x\in\pi(\cQ^+)}x^\perp\cap\cC_X,
\]
where the orthogonal complement $x^\perp$ is taken in $H^{1,1}(X,\bR)$. We define
\[
\mathcal{A}\colonequals \pi(\mathcal{Q}^+)\cap\cC^c.
\]
Notice that automatically $0\notin\mathcal{A}$ because $0\notin\pi(\mathcal{Q}^+)$ since, by hypothesis,
\[
    (H^{2,0}(X)\oplus H^{0,2}(X))\cap H^2(X,\mathbb{Q})=0.
\]
By Lemma~\ref{lemma:denseHyperSec}, we want to show that $\mathcal{A}$ is dense in $\cC^c\setminus\{0\}$ and that
\[
    \mathcal{P}_X = \bigcup_{a\in\mathcal{A}}a^\perp\cap\cC_X.
\]

Let us start with showing that $\mathcal{A}\subseteq\cC^c\setminus\{0\}$ is dense. Consider
\[
    \cQ^+_\bR\colonequals
    \{x\in H^2(X,\bR):x^2>0\}.
\]
Note that $\cQ^+\subseteq\cQ^+_\bR$ is dense.
We claim that
\[
\pi(\cQ^+_\bR)=H^{1,1}(X,\bR).
\]
Indeed, the intersection pairing on
$
    (H^{2,0}(X)\oplus H^{0,2}(X))_\bR
$
has signature $(2,0)$, so this space contains a vector $z$ such that $z^2>0$.
For every $y\in H^{1,1}(X,\bR)$, define
\[
    x\colonequals y+cz,\quad c>0.
\]
Then $\pi(x)=y$ and $x$ lies in $\cQ^+_\bR$ for $c$ sufficiently large.
As a result, $\pi$ restricts to a surjective map
\[\xymatrix{
    \rho\colon \cQ^+_\bR\ar[r] & H^{1,1}(X,\bR).
}\]
It follows that $\rho(\cQ^+)$ is dense in $H^{1,1}(X,\bR)$,
so $\mathcal{A}=\rho(\cQ^+)\cap\cC^c$ is dense in $\cC^c\setminus\{0\}$.

Define $\mathcal{A}^c\colonequals\rho(\cQ^+)\cap\cC$, the complement of $\mathcal{A}$ in $\rho(\cQ^+)$.
Then
\[
    \mathcal{P}_X = \left(
    \bigcup_{a\in\mathcal{A}}a^\perp\cap\cC_X
    \right)\cup\left(
    \bigcup_{b\in\mathcal{A}^c}b^\perp\cap\cC_X
    \right).
\]
The fact that $x\cdot y>0$ for all $x,y\in\cC_X$ \cite{BHPV04}*{IV~Corollary~7.2} implies that the latter subset is in fact empty. Hence
\[
    \mathcal{P}_X = \bigcup_{a\in\mathcal{A}}a^\perp\cap\cC_X
\]
and the density property follows by applying Lemma~\ref{lemma:denseHyperSec}.

To prove the last sentence, first notice that $\cQ^+$ is countable, so $\mathcal{A}$ is countable. Hence $\mathcal{P}_X$ is a union of countably many hyperplane sections.
To see that it contains $\cC_X\cap H^{1,1}(X, \bQ)$, assume that the latter set contains a nonzero element and rescale it as a primitive vector
\[
	w\in\cC_X\cap H^{1,1}(X, \bZ).
\] 
Let $m\colonequals w^2>0$. By Eichler, we can choose an isometry $H^2(X,\bZ)\cong\Lambda_{\Kthree}$ such that
\[
    w = (e+mf,0,0,0,0)\in\Lambda_{\Kthree}
    = U^{\oplus 3}\oplus E_8^{\oplus 2}
\]
where $\{e, f\}$ is a basis for the first copy of $U$ such that $e^2=f^2=0$ and $e\cdot f=1$.
(See \cite{GHS09}*{Proposition~3.3} for a more precise statement and for the reference to Eichler's original work.)
It is easy to find $x\in\Lambda_{\Kthree}$ such that $x^2>0$ and $x\cdot w=0$. For example, one can pick $x = e'+f'$ where $\{e',f'\}$ is the standard basis for the second copy of $U$. As a result, $x$ represents an element in $\cQ^+$ such that $w\in x^\perp\cap\cC_X$. Hence $w\in\mathcal{P}_X$.
\end{proof}

\begin{thm}
\label{thm:denseInKahlerCone}
For every K3 surface $X$, there exists a subset $\mathcal{S}_X$ in the K\"ahler cone $\cK_X$ which is dense in the analytic topology such that
\begin{equation}
\label{eqn:Lag=SLagNonzero}
\Lag(X,\omega) = \SLag(X,\omega)\neq0,
\end{equation}
for every $[\omega]\in\mathcal{S}_X$. Moreover, the following hold:
\begin{itemize}
    \item If $X$ satisfies the hypothesis of Corollary~\ref{cor:rationalPeriodK3}, then $\mathcal{S}_X=\cK_X$.
    \item Otherwise, $\mathcal{S}_X$ is a countable union of hyperplane sections of $\cK_X$. Furthermore, it contains all rational K\"ahler classes, so \eqref{eqn:Lag=SLagNonzero} holds in particular for polarized K3 surfaces.
\end{itemize}
\end{thm}

\begin{proof}
First recall that $\cK_X\subseteq\cC_X$ is a sub-cone.
Due to Proposition~\ref{prop:denseInPosCone}, the collection of $[\omega]\in\cK_X$ such that
$
    \Lag(X,\omega)=[\omega]^\perp\cap H^2(X,\bZ)
$
contains a vector $v$ with $v^2>0$ form a dense subset
\[
	\mathcal{S}_X\colonequals\mathcal{P}_X\cap\cK_X\subseteq\cK_X,
\]
which is a countable union of hyperplane sections and contains all rational K\"ahler classes.
For every $[\omega]\in\mathcal{P}_X\cap\cK_X$, we have $\Lag(X,\omega)=\SLag(X,\omega)$ by Theorem~\ref{thm:LagToSpLag}~\ref{withPositiveLag}, which is a nonzero lattice as it contains a positive vector.
\end{proof}

\begin{rmk}
Let $\check{X}$ be an algebraic K3 surface. Consider the Mukai lattice $\widetilde{H}(\check{X},\bZ)$ and the Mukai vector
\begin{equation}
\label{eqn:mukaiVector}
\xymatrix{
    v\colon
    D^b\mathfrak{Coh}(\check{X})\ar[r] &
    \widetilde{H}(\check{X},\bZ).
}
\end{equation}
Under the choice of a Bridgeland stability condition, every $E\in D^b\mathfrak{Coh}(\check{X})$ admits a Harder--Narasimhan filtration
\[\xymatrix@C=11pt@R=11pt{
    0 = E_0\ar[rr] && E_1\ar[dl]\ar[rr] && E_2\ar[dl]\ar[r] & \cdots\ar[r] & E_{n-1}\ar[rr] && E_n = E\ar[dl]\\
    & A_{1}\ar@{-->}[ul] && A_{2}\ar@{-->}[ul] &&&& A_{n}\ar@{-->}[ul] & 
}\]
where each $A_i$ is a semistable object and thus satisfies $v(A_i)^2\geq-2$ (see, for example, \cite{BM14}*{Theorem~2.15}).
Assume that $(X,\omega)$ is a K\"ahler K3 surface where the homological mirror symmetry
$
    \mathfrak{Fuk}(X)\cong D^b\mathfrak{Coh}(\check{X})
$
is proved, for instance \cite{Sei15} (see also \cite{She19}*{Remark~1.17}), and such that $\check{X}$ is constructed over $\bC$. Notice that usually the mirrors are constructed over non-archimedean fields. Assume further that we have a commutative diagram as follows
\[\xymatrix{
    \mathfrak{Fuk}(X)\ar[d]\ar[r]^-\sim &  D^b\mathfrak{Coh}(\check{X})\ar[d]^-v\\
    [\omega]^{\perp }\ar[r]^-\sim & \widetilde{H}^{1,1}(\check{X},\bZ),
}\]
where $[\omega]^{\perp}$ is the orthogonal complement taken inside $H^2(X,\mathbb{Z})$. Since the Mukai vector \eqref{eqn:mukaiVector} factors through the numerical Grothendieck group $K_\text{num}(\check{X})$ and induces an isometry
\[\xymatrix{
    K_\text{num}(\check{X})\ar[r]^-\sim
    & \widetilde{H}^{1,1}(\check{X}, \bZ)
}\]
(see \cite{AT14}*{\S2.2}), one can prove Theorem~\ref{thm:LagToSpLag}~\ref{withPositiveLag} for such an $X$.
\end{rmk}

%--------------------Special Lagrangian fibrations
\subsection{Special Lagrangian fibrations on K3 surfaces}
\label{subsect:appendix}

We now turn to the analysis of the K\"ahler classes in the K\"ahler cone of a K3 surface for which we have a special Lagrangian torus fibration.

\begin{thm}
\label{thm:SYZ}
The set $\mathcal{S}_X$ in Theorem~\ref{thm:denseInKahlerCone} contains a subset $\mathcal{T}_X$ which is dense in $\cK_X$ in the analytic topology such that every K\"ahler form $\omega$ from $\mathcal{T}_X$ guarantees the existence of a special Lagrangian fibration on $X$. Moreover, the following hold:
\begin{itemize}
    \item If $X$ satisfies the hypothesis of Corollary~\ref{cor:rationalPeriodK3}, then $\mathcal{T}_X$ is a union of countably many hyperplane sections in $\cK_X$.
    \item Otherwise, $\mathcal{T}_X$ is a union of countably many codimension two linear sections in $\cK_X$.
\end{itemize}
In both cases, $\mathcal{T}_X$ contains all rational K\"ahler classes. In particular, every polarized K3 surface admits a special Lagrangian fibration. 
\end{thm}

The proof of the theorem is based on the following proposition which is similar to Proposition~\ref{prop:denseInPosCone}.

\begin{prop}
\label{prop:denseInPosConeIso}
On a K3 surface $X$, the classes $w\in\cC_X$ whose orthogonal complement $w^\perp\cap H^2(X,\bZ)$ contains a nonzero vector $\ell$ with $\ell^2=0$ form a dense subset $\mathcal{R}_X\subseteq\cC_X$ in the analytic topology.
This subset is a union of countably many hyperplane sections and it contains $\cC_X\cap H^{1,1}(X, \bQ)$.
\end{prop}

\begin{proof}
The proof is analogous to the proof of Proposition~\ref{prop:denseInPosCone}, but with the following changes. First of all, replace $\mathcal{Q}^+$ with
\[
	\cQ\colonequals\{x\in H^2(X,\bQ):x^2=0\}\setminus\{0\},
\]
so that 
\[
	\mathcal{R}_X
	=\bigcup_{\ell\in\cQ}
	\ell^\perp\cap\cC_X=\bigcup_{\ell\in\pi(\cQ)}
	\ell^\perp\cap\cC_X.
\]
Since $(H^{2,0}(X)\oplus H^{0,2}(X))_\bQ$ is either empty or positive definite, we have $0\notin\pi(\mathcal{Q})$, which implies that $\mathcal{B}\colonequals\pi(\mathcal{Q})\cap\mathcal{C}^c$ does not contain $\{0\}$.

To show that $\mathcal{B}\subseteq\cC^c\setminus\{0\}$ is dense, consider
\[
    \cQ_\bR\colonequals
    \{x\in H^2(X,\bR):x^2=0\}\setminus\{0\}.
\]
Note that $\cQ\subseteq\cQ_\bR$ is dense.
We claim that
\[
\pi(\cQ_\bR)=H^{1,1}(X,\bR).
\]
Indeed, the intersection pairing on
$
    (H^{2,0}(X)\oplus H^{0,2}(X))_\bR
$
has signature $(2,0)$, so this space contains a vector $z$ such that $z^2>0$.
For every $y\in H^{1,1}(X,\bR)$, define
\[
    x\colonequals y-cz,
\]
where $c\colonequals\sqrt{y^2/z^2}\in\bR$. Then $\pi(x)=y$ and $x$ lies in $\cQ_\bR$.
As a result, $\pi$ restricts to a surjective map
\[\xymatrix{
    \eta\colon \cQ_\bR\ar[r] & H^{1,1}(X,\bR).
}\]
It follows that $\eta(\cQ)$ is dense in $H^{1,1}(X,\bR)$,
so $\mathcal{B}=\eta(\cQ)\cap\cC^c$ is dense in $\cC^c\setminus\{0\}$.

Define $\mathcal{B}^c\colonequals\eta(\cQ)\cap\cC$, the complement of $\mathcal{B}$ in $\eta(\cQ)$.
Then
\[
    \mathcal{R}_X = \left(
    \bigcup_{a\in\mathcal{B}}a^\perp\cap\cC_X
    \right)\cup\left(
    \bigcup_{b\in\mathcal{B}^c}b^\perp\cap\cC_X
    \right)
    = \bigcup_{a\in\mathcal{B}}a^\perp\cap\cC_X.
\]
The second equality follows from the fact that $x\cdot y>0$ for all $x,y\in\cC_X$ \cite{BHPV04}*{IV~Corollary~7.2}. By applying Lemma~\ref{lemma:denseHyperSec} again, we conclude that $\mathcal{R}_X$ is dense in $\cC_X$.

Finally, $\mathcal{R}_X$ contains all the rational K\"ahler classes $[\omega]$ since we can always find a nonzero isotropic vector in $[\omega]^\perp$ (see the end of the proof of Proposition~\ref{prop:denseInPosCone}).
\end{proof}

\begin{lemma}
\label{lemma:denseCodimTwoSec}
Let $\cC^c$ be the complement of $\cC$ in $H^{1,1}(X,\bR)$.
Suppose that $\mathcal{A},\mathcal{B}\subseteq\cC^c\setminus\{0\}$ are disjoint dense subsets.
Then the union
\[
    \bigcup_{(a,b)\in\mathcal{A}\times\mathcal{B}}a^\perp\cap b^\perp\cap\mathcal{C}_X
\]
is dense in $\cC_X$.
\end{lemma}

\begin{proof}
First of all, let $\Delta\subseteq(\cC^c\setminus\{0\})^{\times2}$ denote the diagonal and define
\[
    \mathcal{E}\colonequals(\cC^c\setminus\{0\})^{\times2}\setminus\Delta.
\]
We claim that
\begin{equation}
\label{eqn:coneInCodimTwo}
    \cC_X\subseteq\bigcup_{(a,b)\in\mathcal{E}}a^\perp\cap b^\perp.
\end{equation}
Indeed, for every $x\in\cC_X$, we have $x^\perp\cap\cC_X=\emptyset$ by \cite{BHPV04}*{IV~Corollary~7.2}. This implies that $x^\perp\subseteq\cC^c$, as it was already explained in the proof of Lemma~\ref{lemma:denseHyperSec}. So there exist $a,b\in\overline{\cC}^c$, $a\neq b$ such that $a\cdot x=0$ and $b\cdot x=0$, implying the containment in \eqref{eqn:coneInCodimTwo}.

By \eqref{eqn:coneInCodimTwo}, it is sufficient to show that
\begin{equation}
\label{eqn:denseHyperInCodimTwo}
    \bigcup_{(a,b)\in\mathcal{A}\times\mathcal{B}}a^\perp\cap b^\perp
    \subseteq
    \bigcup_{(a,b)\in\mathcal{E}}a^\perp\cap b^\perp
\end{equation}
is a dense subset.
For simplicity of notation, let $V\colonequals H^{1,1}(X,\bR)$. Consider the diagram
\[\xymatrix{
    & \mathcal{U}\colonequals\{(a,b,z)\in\mathcal{E}\times V : a^\perp\cap b^\perp\ni z\}\ar[dl]_-{\pi_{12}}\ar[dr]^-{\pi_3} &\\
    \mathcal{E} && V.
}\]
Then $\pi_{12}$ is a vector bundle fibered in $18$-dimensional real vector spaces, and the image of $\pi_3$ coincides with the left hand side of \eqref{eqn:denseHyperInCodimTwo}. In particular, $\pi_{12}$ is an open map. 
Because $\mathcal{A}\times\mathcal{B}$ is dense in $\mathcal{E}$, $\pi_{12}^{-1}(\mathcal{A}\times\mathcal{B})$ is dense in $\mathcal{U}$ and thus
\[
    \bigcup_{(a,b)\in\mathcal{A}\times\mathcal{B}}a^\perp\cap b^\perp
    = \pi_3(\pi_{12}^{-1}(\mathcal{A}\times\mathcal{B}))
\]
is dense in the image of $\pi_3$, which completes the proof.
\end{proof}

\begin{proof}[Proof of Theorem~\ref{thm:SYZ}]
Given $[\omega]\in\cK_X$, assume there exist $\ell\in[\omega]^\perp$ such that $\ell$ is a nonzero isotropic vector. Then by Lemma~\ref{lemma:holo-HK-Lag}~\ref{lag-To-Holo}, we can perform a hyperk\"ahler rotation to make $\ell$ a $(1,1)$-class, which still satisfies $\ell^2=0$.
The linear system $|L|$ given by $\ell$ may contain smooth rational curves as base locus, which can be removed by applying suitable reflections with respect to $(-2)$-curves. (Notice that these are Hodge isometries.)
Now $|L|$ induces an elliptic fibiration. By reversing the hyperk\"ahler rotation, the fibration turns into a special Lagrangian torus fibration on $X$ with respect to $\omega$ (this strategy is well know, and was also used in \cite{OO21}). So what we need to do is to study the distribution of K\"ahler classes for which such vector $\ell$ exist.

Suppose that we are in the situation of Corollary~\ref{cor:rationalPeriodK3}. Retain the notation from Proposition~\ref{prop:denseInPosConeIso} and define
\[
    \mathcal{T}_X\colonequals
    \mathcal{R}_X\cap\mathcal{K}_X.
\]
In this case, we have $\mathcal{T}_X\subseteq\mathcal{S}_X$ since $\mathcal{S}_X = \cK_X$ by Theorem~\ref{thm:denseInKahlerCone}.
By definition, every $[\omega]\in\mathcal{T}_X$ lies in $\mathcal{R}_X$, so there exists a nonzero isotropic vector $\ell\in[\omega]^\perp$ by Proposition~\ref{prop:denseInPosConeIso}.

Now assume that the hypotheses of Corollary~\ref{cor:rationalPeriodK3} are not satisfied. 
Consider the two sets
\[
    \cQ^+\colonequals\{x\in H^2(X,\bQ):x^2>0\},\;
    \cQ\colonequals\{x\in H^2(X,\bQ):x^2=0\}\setminus\{0\},
\]
project them into $H^{1,1}(X,\bR)$ using the map
\[\xymatrix{
    \pi\colon H^2(X,\bR)\ar[r] & H^{1,1}(X,\bR),
}\]
and let $\mathcal{A}$ and $\mathcal{B}$, respectively, be their intersections with $\cC^c\setminus\{0\}$.
From the proofs of Propositions~\ref{prop:denseInPosCone} and \ref{prop:denseInPosConeIso}, we know that both $\mathcal{A}$ and $\mathcal{B}$ are dense in $\cC^c\setminus\{0\}$.
We claim that $\mathcal{A}$ and $\mathcal{B}$ are disjoint from each other.
Suppose not, that is, there exists
\[
    x\in \mathcal{A}\cap\mathcal{B}
    \subseteq H^{1,1}(X,\bR).
\]
So $\pi^{-1}(x)$ contains nonzero vectors $v,\ell\in H^2(X,\bQ)$ satisfying $v^2>0$ and $\ell^2=0$.
In this case, we can write
\[
    v = x + s
    \quad\text{and}\quad
    \ell = x + t
\]
for some $s,t\in(H^{2,0}(X)\oplus H^{0,2}(X))_\bR$.
Take the subtraction
\[
    \xi\colonequals
    v-\ell = (x+s) - (x+t) = s-t.
\]
Note that $(v-\ell)\in H^2(X,\bQ)\setminus\{0\}$ and that $(s-t)\in H^{2,0}(X)\oplus H^{0,2}(X)$.
Hence $\xi$ is a nonzero vector in
\[
    (H^{2,0}(X)\oplus H^{0,2}(X))\cap H^2(X,\bQ).
\]
But this contradicts to our hypothesis.

Now we can apply Lemma~\ref{lemma:denseCodimTwoSec} to conclude that the set
\[
\bigcup_{(a,b)\in\mathcal{A}\times\mathcal{B}}a^\perp\cap b^\perp\cap\mathcal{C}_X
\]
is dense in $\mathcal{C}_X$. Define
\[
\mathcal{T}_X\colonequals\bigcup_{(a,b)\in\mathcal{A}\times\mathcal{B}}a^\perp\cap b^\perp\cap\mathcal{K}_X
\]
which is dense in $\mathcal{K}_X$.
We note that this is a countable union of codimension two linear subspaces in $\mathcal{S}_X$, and that it contains all the rational K\"ahler classes.
The latter is true because we can find two nonzero vectors $v,\ell$ in $[\omega]^\perp$ such that $v^2>0$ and $\ell^2=0$ (see the end of the proof of Proposition~\ref{prop:denseInPosCone}).
\end{proof}

The following is an immediate application.

\begin{cor}
\label{cor:3ToriInCY3}
Let $Y$ be a Fano $3$-fold and $D\in |-K_Y|$ be a smooth anti-canonical divisor. Then the log Calabi--Yau $3$-fold $X=Y\backslash D$ contains infinitely many special Lagrangian $3$-tori.
\end{cor}

\begin{proof}
Since $Y$ is Fano $3$-fold, $D$ is an algebraic K3 surface \cite{Bea04}. Then $D$ admits a special Lagrangian fibration by Theorem~\ref{thm:SYZ}. Let $L\subseteq D$ be any of the special Lagrangian $2$-torus. By \cite{CJL21}*{Theorem~1.1}, there exist infinitely many special Lagrangians in $X$ diffeomorphic to $S^1\times L\cong T^3$.
\end{proof}

\bigskip
\bibliography{SpLagK3_bib}

\begin{bibdiv}
\begin{biblist}

\bib{AT14}{article}{
      author={Addington, Nicolas},
      author={Thomas, Richard},
       title={Hodge theory and derived categories of cubic fourfolds},
        date={2014},
        ISSN={0012-7094},
     journal={Duke Math. J.},
      volume={163},
      number={10},
       pages={1885\ndash 1927},
         url={https://doi.org/10.1215/00127094-2738639},
      review={\MR{3229044}},
}

\bib{Bea04}{incollection}{
      author={Beauville, Arnaud},
       title={Fano threefolds and {$K3$} surfaces},
        date={2004},
   booktitle={The {F}ano {C}onference},
   publisher={Univ. Torino, Turin},
       pages={175\ndash 184},
      review={\MR{2112574}},
}

\bib{BHPV04}{book}{
      author={Barth, Wolf~P.},
      author={Hulek, Klaus},
      author={Peters, Chris A.~M.},
      author={Ven, Antonius Van~de},
       title={Compact complex surfaces},
     edition={Second},
      series={Ergebnisse der Mathematik und ihrer Grenzgebiete. 3. Folge. A
  Series of Modern Surveys in Mathematics},
   publisher={Springer-Verlag, Berlin},
        date={2004},
      volume={4},
        ISBN={3-540-00832-2},
         url={https://doi.org/10.1007/978-3-642-57739-0},
      review={\MR{2030225}},
}

\bib{BM14}{article}{
      author={Bayer, Arend},
      author={Macr\`\i, Emanuele},
       title={M{MP} for moduli of sheaves on {K}3s via wall-crossing: nef and
  movable cones, {L}agrangian fibrations},
        date={2014},
        ISSN={0020-9910},
     journal={Invent. Math.},
      volume={198},
      number={3},
       pages={505\ndash 590},
         url={https://doi.org/10.1007/s00222-014-0501-8},
      review={\MR{3279532}},
}

\bib{CJL21}{article}{
      author={Collins, Tristan~C.},
      author={Jacob, Adam},
      author={Lin, Yu-Shen},
       title={Special {L}agrangian submanifolds of log {C}alabi-{Y}au
  manifolds},
        date={2021},
        ISSN={0012-7094},
     journal={Duke Math. J.},
      volume={170},
      number={7},
       pages={1291\ndash 1375},
         url={https://doi.org/10.1215/00127094-2021-0012},
      review={\MR{4255060}},
}

\bib{GH94}{book}{
      author={Griffiths, Phillip},
      author={Harris, Joseph},
       title={Principles of algebraic geometry},
      series={Wiley Classics Library},
   publisher={John Wiley \& Sons, Inc., New York},
        date={1994},
        ISBN={0-471-05059-8},
         url={https://doi-org.ezproxy.bu.edu/10.1002/9781118032527},
        note={Reprint of the 1978 original},
      review={\MR{1288523}},
}

\bib{GHS09}{article}{
      author={Gritsenko, Valery},
      author={Hulek, Klaus},
      author={Sankaran, Gregory},
       title={Abelianisation of orthogonal groups and the fundamental group of
  modular varieties},
        date={2009},
     journal={J. Algebra},
      volume={322},
       pages={463\ndash 478},
}

\bib{GW00}{article}{
      author={Gross, Mark},
      author={Wilson, P. M.~H.},
       title={Large complex structure limits of {$K3$} surfaces},
        date={2000},
        ISSN={0022-040X},
     journal={J. Differential Geom.},
      volume={55},
      number={3},
       pages={475\ndash 546},
         url={http://projecteuclid.org/euclid.jdg/1090341262},
      review={\MR{1863732}},
}

\bib{Hat14}{article}{
      author={Hattori, Kota},
       title={New examples of compact special {L}agrangian submanifolds
  embedded in hyper-{K}\"{a}hler manifolds},
        date={2019},
        ISSN={1527-5256},
     journal={J. Symplectic Geom.},
      volume={17},
      number={2},
       pages={301\ndash 335},
         url={https://doi-org.ezproxy.bu.edu/10.4310/JSG.2019.v17.n2.a1},
      review={\MR{3992718}},
}

\bib{HL82}{article}{
      author={Harvey, Reese},
      author={Lawson, H.~Blaine, Jr.},
       title={Calibrated geometries},
        date={1982},
        ISSN={0001-5962},
     journal={Acta Math.},
      volume={148},
       pages={47\ndash 157},
         url={https://doi-org.ezproxy.bu.edu/10.1007/BF02392726},
      review={\MR{666108}},
}

\bib{Huy16}{book}{
      author={Huybrechts, Daniel},
       title={Lectures on {K}3 surfaces},
      series={Cambridge Studies in Advanced Mathematics},
   publisher={Cambridge University Press, Cambridge},
        date={2016},
      volume={158},
        ISBN={978-1-107-15304-2},
         url={https://doi.org/10.1017/CBO9781316594193},
      review={\MR{3586372}},
}

\bib{Joy15}{article}{
      author={Joyce, Dominic},
       title={Conjectures on {B}ridgeland stability for {F}ukaya categories of
  {C}alabi-{Y}au manifolds, special {L}agrangians, and {L}agrangian mean
  curvature flow},
        date={2015},
        ISSN={2308-2151},
     journal={EMS Surv. Math. Sci.},
      volume={2},
      number={1},
       pages={1\ndash 62},
         url={https://doi.org/10.4171/EMSS/8},
      review={\MR{3354954}},
}

\bib{Nev07}{article}{
      author={Neves, Andr\'{e}},
       title={Singularities of {L}agrangian mean curvature flow: zero-{M}aslov
  class case},
        date={2007},
        ISSN={0020-9910},
     journal={Invent. Math.},
      volume={168},
      number={3},
       pages={449\ndash 484},
         url={https://doi.org/10.1007/s00222-007-0036-3},
      review={\MR{2299559}},
}

\bib{OO21}{book}{
      author={Odaka, Yuji},
      author={Oshima, Yoshiki},
       title={Collapsing {K}3 surfaces, tropical geometry and moduli
  compactifications of {S}atake, {M}organ-{S}halen type},
      series={MSJ Memoirs},
   publisher={Mathematical Society of Japan, Tokyo},
        date={2021},
      volume={40},
        ISBN={978-4-86497-104-1},
         url={https://doi.org/10.1142/e071},
      review={\MR{4290084}},
}

\bib{Sai74}{article}{
      author={Saint-Donat, B.},
       title={Projective models of {$K-3$} surfaces},
        date={1974},
        ISSN={0002-9327},
     journal={Amer. J. Math.},
      volume={96},
       pages={602\ndash 639},
         url={https://doi.org/10.2307/2373709},
      review={\MR{364263}},
}

\bib{Sei15}{article}{
      author={Seidel, Paul},
       title={Homological mirror symmetry for the quartic surface},
        date={2015},
     journal={Mem. Amer. Math. Soc.},
      volume={236},
      number={1116},
       pages={vi+129},
}

\bib{She19}{incollection}{
      author={Sheridan, Nick},
       title={Versality in mirror symmetry},
        date={2019},
   booktitle={Current developments in mathematics 2017},
   publisher={Int. Press, Somerville, MA},
       pages={37\ndash 86},
      review={\MR{3971325}},
}

\bib{Smo00}{book}{
      author={Smoczyk, Knut},
       title={Der {Lagrangesche} mittlere {Kr{\"u}mmungsfluss}},
   publisher={Leipzig: Univ. Leipzig (Habil.-Schr.)},
        date={2000},
}

\bib{SW01}{article}{
      author={Schoen, Richard},
      author={Wolfson, Jon},
       title={Minimizing area among {L}agrangian surfaces: the mapping
  problem},
        date={2001},
        ISSN={0022-040X},
     journal={J. Differential Geom.},
      volume={58},
      number={1},
       pages={1\ndash 86},
         url={http://projecteuclid.org/euclid.jdg/1090348282},
      review={\MR{1895348}},
}

\bib{SYZ96}{article}{
      author={Strominger, Andrew},
      author={Yau, Shing-Tung},
      author={Zaslow, Eric},
       title={Mirror symmetry is {$T$}-duality},
        date={1996},
     journal={Nuclear Phys. B},
      volume={479},
      number={1-2},
       pages={243\ndash 259},
}

\bib{TY02}{article}{
      author={Thomas, Richard~P.},
      author={Yau, Shing-Tung},
       title={Special {L}agrangians, stable bundles and mean curvature flow},
        date={2002},
        ISSN={1019-8385},
     journal={Comm. Anal. Geom.},
      volume={10},
      number={5},
       pages={1075\ndash 1113},
         url={https://doi.org/10.4310/CAG.2002.v10.n5.a8},
      review={\MR{1957663}},
}

\bib{Wol05}{article}{
      author={Wolfson, Jon},
       title={Lagrangian homology classes without regular minimizers},
        date={2005},
        ISSN={0022-040X},
     journal={J. Differential Geom.},
      volume={71},
      number={2},
       pages={307\ndash 313},
         url={http://projecteuclid.org/euclid.jdg/1143651771},
      review={\MR{2197143}},
}

\bib{Y78}{article}{
      author={Yau, Shing-Tung},
       title={On the {R}icci curvature of a compact {K}\"{a}hler manifold and
  the complex {M}onge-{A}mp\`ere equation. {I}},
        date={1978},
        ISSN={0010-3640},
     journal={Comm. Pure Appl. Math.},
      volume={31},
      number={3},
       pages={339\ndash 411},
         url={https://doi-org.ezproxy.bu.edu/10.1002/cpa.3160310304},
      review={\MR{480350}},
}

\end{biblist}
\end{bibdiv}
\bibliographystyle{alpha}

\ContactInfo

\end{document}